\documentclass[11pt]{article}
\usepackage[margin=1.15in]{geometry}
\usepackage{amsmath,amssymb,amsthm,mathtools}
\usepackage{booktabs}
\usepackage{microtype}
\usepackage[colorlinks,citecolor=blue,linkcolor=blue,urlcolor=blue]{hyperref}

\newtheorem{theorem}{Theorem}
\newtheorem{proposition}[theorem]{Proposition}
\newtheorem{lemma}[theorem]{Lemma}
\newtheorem{corollary}[theorem]{Corollary}
\newtheoremstyle{conjectureplain}
  {\topsep}{\topsep}{\itshape}{}{\bfseries}{\textnormal{.}}{5pt plus 1pt minus 1pt}{}
\theoremstyle{conjectureplain}
\newtheorem{conjecture}[theorem]{Conjecture}
\theoremstyle{definition}
\newtheorem{remark}[theorem]{Remark}
\newtheorem{problem}[theorem]{Problem}

\DeclareMathOperator{\Aut}{Aut}
\newcommand{\Lk}{\mathcal{L}}
\newcommand{\bin}[2]{\binom{#1}{#2}}
\newcommand{\Kbal}{K_{\lfloor n/2\rfloor,\lceil n/2\rceil}}

\title{The minimum of the graph likelihood}

\author{
  Simone Severini \\
 UCL
  \and
 Eric W. Weisstein \\
  Wolfram Research
  }
\date{Revised October 28, 2026}

\begin{document}
\maketitle

\begin{abstract}
The likelihood of a finite simple undirected graph $G$ on $n$ vertices is the probability that the uniform sequential
attachment process, which at each step joins a new vertex to a uniformly random subset of
uniformly random size of the vertices already present, outputs a graph isomorphic to $G$.
Dervovic, Mocherla and Severini conjectured that the likelihood is minimised by the balanced
complete bipartite graph. We prove that, among complete bipartite graphs of a given order,
the balanced one uniquely minimises the likelihood. Exact computation shows that
it also minimises over all graphs for every order from $6$ through $14$. At $n=15$,
the exact global minimisers are $C_5[\bar K_3]$ and its complement, giving
the first counterexample. Completed exact searches show that the global minimisers at order $16$
are $K_{8,8}$ and its complement. At order $17$, the unique minimiser is a self-complementary
mixed blow-up of $C_5$. At order $18$, the minimisers are $K_{9,9}$ and its complement.
We show that the failure
is not sporadic by proving that the likelihood of the balanced complete bipartite graph is
$2^{-(1/2-1/(8\ln 2)+o(1))n^2}$, whereas the minimum over all graphs of order $n$ is
$2^{-(1/2+o(1))n^2}$, so the conjectured minimiser exceeds the minimum by a factor exponential
in $n^2$. We also determine the Shannon entropy of the process to leading order, namely
$n^2/(4\ln 2)$ bits, which shows that the conjectured minimiser is in fact more likely than a
typical output of the process. The proofs rest on a vertex deletion recurrence which evaluates
the likelihood in time $O(2^n n)$ and which closes on the blow-ups of any fixed base graph.
\end{abstract}

\section{Introduction}

Let $G_1=K_1$ and construct a random graph $G_t$ from $G_{t-1}$ in three substeps: (1) choose $k$ uniformly from $\{0,1,\dots,t-1\}$; (2) choose a $k$-element subset $W$ of the vertex set $V(G_{t-1})$ uniformly; and (3) add a new vertex adjacent to exactly the vertices of $W$. For a finite simple undirected graph $G$ on $n$ vertices, its \emph{likelihood} is the probability that $G_n$ is isomorphic to $G$:
\[
\Lk(G)\;=\;\Pr\big[G_n\cong G\big].
\]
This invariant was introduced by Banerji, Mansour and Severini \cite{BMS14} as a graph-theoretic analogue of the infinite monkey theorem, the underlying process being the uniform case of a family of sequential attachment models studied by Janson and Severini \cite{JS13} in the context of graph limits. The likelihood is a probability distribution on the isomorphism classes of graphs of order $n$, so $\sum_G\Lk(G)=1$, where the sum runs over one representative of each isomorphism class of graphs with $n$ vertices. It quantifies how easily a graph is produced by a mechanism that uses no information about the graph built so far. Writing $\overline G$ for the complement of $G$, explicit values include $\Lk(K_n)=\Lk(\bar K_n)=1/n!$ and $\Lk(K_{1,n-1})=n(n!)^{-2}\sum_{i=0}^{n-1}i!$, and the invariant satisfies $\Lk(G)=\Lk(\overline G)$ together with the bounds \cite[Cor.~18]{BMS14}
\begin{equation}\label{eq:BMSbounds}
\frac{1}{\lvert\Aut(G)\rvert\prod_{i=1}^{n}\bin{i-1}{\lfloor (i-1)/2\rfloor}}\;\le\;\Lk(G)\;\le\;\frac{1}{\lvert\Aut(G)\rvert}.
\end{equation}

Two extremal problems were left open in \cite{BMS14}. Weisstein \cite{MW} computed $\Lk(G)$ for every graph of order at most $10$ in 2013 and observed that for $n\le 10$, with the exceptions $n=3$ and $n=5$, the minimum is attained by $\Kbal$ and its complement. The minimum values are recorded by numerator and denominator in \href{https://oeis.org/A234234}{A234234} and \href{https://oeis.org/A234235}{A234235}, and are listed as fractions for $4\le n\le18$ in Table~\ref{tab:min} below. (The analogous maximum numerator and denominator sequences are \href{https://oeis.org/A234236}{A234236} and \href{https://oeis.org/A234237}{A234237} \cite{OEIS}.) On this numerical evidence Dervovic, Mocherla and Severini \cite[\S5]{DMS18} proposed the following assertion. Since $n=3$ and $n=5$ were already known exceptions, we understand their intended claim to apply to the nonexceptional range $n\ge6$. It is this claim that we refute below.

\begin{conjecture}[\cite{DMS18}; refuted]\label{conj:DMS}
For $n\ge6$, if $n=2p$ the minimum of $\Lk$ over graphs of order $n$ is attained by $K_{p,p}$, and if $n=2p-1$ it is attained by $K_{p-1,p}$.
\end{conjecture}

The computations below determine all four sequences at $n=11$ and $n=12$
and the two minimum sequences through $n=18$. They verify
Conjecture~\ref{conj:DMS} through $n=14$. Our central finite result is that
the conjecture is false: we determine the exact global minimum at $n=15$,
showing that this is its first failure. We also prove that counterexamples
exist at every sufficiently large order. The completed order-$16$ search
shows that $K_{8,8}$ and its complement are the only minimisers at that order.
At order $17$, the unique minimiser is the self-complementary graph $H_{17}$
defined in Proposition~\ref{prop:min17}; it improves on the independent
$C_5$ blow-up that supplied the starting witness for that search.
The completed order-$18$ search gives $K_{9,9}$ and its complement.

The heuristic offered in \cite{DMS18} is that these graphs have a relatively large automorphism group, which by \eqref{eq:BMSbounds} pushes $\Lk$ down, and that by Mantel's theorem they are the triangle-free graphs with the largest number of edges. Our purpose is to determine how far this reasoning reaches. The results below show both its force and its limitations.

Our first results are positive. Section~\ref{sec:bip} settles the conjecture inside the family in which it is stated.

\begin{theorem}\label{thm:bip}
For every $n\ge 4$ the function $a\mapsto \Lk(K_{a,n-a})$, $0\le a\le \lfloor n/2\rfloor$, attains its unique minimum at $a=\lfloor n/2\rfloor$.
\end{theorem}

The 2013 enumeration already covered all orders through $n=10$.
Section~\ref{sec:small} adds $n=11$ and $n=12$, thereby verifying the
conjecture for $6\le n\le 12$.  Section~\ref{sec:min13-14} adds the remaining
cases $n=13$ and $n=14$ using the canonical signed twin-family reduction
and exact automorphism bounds.  We regard
a self-contained verification in exact
arithmetic as worth recording; all our computations use integer or rational
arithmetic only.

\begin{theorem}\label{thm:small}
For $4\le n\le 12$ the minimum of $\Lk$ over graphs of order $n$ is attained exactly by the two isomorphism classes $\Kbal$ and its complement, except for $n=5$, where the unique minimiser is $C_5$. The minimum values are listed in Table~\ref{tab:min}.
\end{theorem}

\begin{table}[htbp]
\centering
\begin{tabular}{@{}rll@{}}
\toprule
$n$ & $\min_G\Lk(G)$ & minimiser\\
\midrule
$4$ & $1/36$ & $K_{2,2}$\\
$5$ & $1/270$ & $C_5$\\
$6$ & $23/259200$ & $K_{3,3}$\\
$7$ & $319/54432000$ & $K_{3,4}$\\
$8$ & $319/15240960000$ & $K_{4,4}$\\
$9$ & $76441/115221657600000$ & $K_{4,5}$\\
$10$ & $76441/145179288576000000$ & $K_{5,5}$\\
$11$ & $20692621/2012184939663360000000$ & $K_{5,6}$\\
$12$ & $20692621/11155553305493667840000000$ & $K_{6,6}$\\
$13$ & $368997387881/16080060756670792571289600000000$ & $K_{6,7}$\\
$14$ & $368997387881/386307379618259120732661350400000000$ & $K_{7,7}$\\
$15$ & $63977511069907/43503039787261205233506826321920000000000$ & $C_5[\bar K_3]$\\
$16$ & $7628328998218493/107497751435913928580204708114517196800000000000$ & $K_{8,8}$\\[1ex]
$17$ & $\dfrac{46676188724227801}{671983801047531669544656802525006213939200000000000}$ & $H_{17}$\\[2ex]
$18$ & $\dfrac{2200453451294991512083}{2881657382291313298001639051759143747136048332800000000000000}$ & $K_{9,9}$\\[1ex]
\bottomrule
\end{tabular}
\caption{Minimum likelihoods among all simple graphs on $n$ vertices. Minimisers are unique up to complementation. Values through $n=10$ were computed in 2013 \cite{MW,OEIS}; the values for $n=11$ through $n=18$ are determined here. The graph $H_{17}$ is self-complementary and is defined in Proposition~\ref{prop:min17}.}
\label{tab:min}
\end{table}

The two remaining smaller orders can be settled without enumerating all graphs,
by an exact enumeration of the graphs whose automorphism groups are large enough
to meet the lower bound in \eqref{eq:BMSbounds}; details are given in
Section~\ref{sec:min13-14}.

\begin{theorem}\label{thm:firstfailure}
For $n=13$ and $n=14$, the minimum of $\Lk$ is attained exactly by $\Kbal$
and its complement.
\end{theorem}

Direct exhaustive search is feasible only through $n=12$, but the first
counterexample occurs at $n=15$. Write $H[\bar K_m]$ for the blow-up of $H$ in which every
vertex is replaced by an independent set of size $m$, adjacency being inherited.

\begin{theorem}\label{thm:counter}
The minimum likelihood among all graphs on $15$ vertices is
\[
\min_G\Lk(G)=\frac{63977511069907}{43503039787261205233506826321920000000000}
=1.4706\ldots\times 10^{-27}.
\]
It is attained exactly by $C_5[\bar K_3]$ and its complement. In
particular, Conjecture~\ref{conj:DMS} first fails at $n=15$, since
\[
\Lk(K_{7,8})=\frac{7628328998218493}{1044072954894268925604163831726080000000000}=7.3063\ldots\times 10^{-27},
\]
so
\[
\frac{\Lk(C_5[\bar K_3])}{\Lk(K_{7,8})}
=\frac{1535460265677768}{7628328998218493}=0.20128\ldots<1.
\]
\end{theorem}

The witness is instructive. It is a triangle-free vertex-transitive graph, so it does not contradict the first half of the heuristic of \cite{DMS18}. However, it does contradict the second half, since $C_5[\bar K_3]$ has $45$ edges against the $56$ of $K_{7,8}$. Blow-ups of $C_5$ are natural candidates: they remain triangle-free, nearly balanced blow-ups retain substantial symmetry, and the underlying odd cycle keeps them far from bipartite. It is precisely this family that overtakes the complete bipartite graphs.

Section~\ref{sec:asym} explains why counterexamples must eventually occur and shows that Theorem~\ref{thm:counter} is the visible edge of an asymptotic phenomenon. Automorphism-group size contributes only $O(n\log n)$ to the logarithm of the likelihood, whereas the difference between the conjectured value and the minimum is of order $n^2$. Symmetry therefore cannot make the balanced complete bipartite graph a global minimiser for large $n$.

\begin{theorem}\label{thm:asym}
As $n\to\infty$,
\[
\log_2\frac{1}{\Lk(\Kbal)}=\Big(\frac12-\frac1{8\ln 2}\Big)n^2+O(n\log n)
=(0.31966\ldots)\,n^2+O(n\log n),
\]
while
\[
\log_2\frac{1}{\min_G\Lk(G)}=\frac{n^2}{2}+O(n\log n).
\]
Consequently $\Lk(\Kbal)\big/\min_G\Lk(G)=2^{(1/(8\ln 2)+o(1))n^2}$, and Conjecture~\ref{conj:DMS} fails for every sufficiently large $n$.
\end{theorem}

Section~\ref{sec:entropy} completes the picture by computing the entropy of the process, which identifies the likelihood of a typical output and shows that the conjectured minimiser sits above it.

\begin{theorem}\label{thm:entropy}
The Shannon entropy of the distribution $\Lk$ on isomorphism classes of graphs of order $n$ equals $\frac{n^2}{4\ln 2}+O(n\log n)=(0.36067\ldots)n^2+O(n\log n)$ bits, against the $\frac{n^2}{2}+O(n\log n)$ bits of the uniform distribution on isomorphism classes.
Moreover, if $G_n$ is the output of the process, then
\[
-\frac{1}{n^2}\log_2\Lk(G_n)\;\xrightarrow{\mathrm p}\;\frac{1}{4\ln 2}.
\]
\end{theorem}

Since $\tfrac12-\tfrac{1}{8\ln 2}<\tfrac{1}{4\ln 2}$, the graph $\Kbal$ is asymptotically \emph{more} likely than a typical output of the process, which is a second way of seeing that it cannot be a minimiser. Section~\ref{sec:remarks} collects remarks on the maximum, where we show that no analogue of Theorem~\ref{thm:bip} is available, and open problems.

The tool behind all of this is a vertex deletion recurrence for $\Lk$, stated in \S\ref{sec:prelim}. It appears in this form in \cite{MW} and we claim no novelty for it; what we add is the observation that it closes on the blow-ups of any fixed base graph, which is what makes both Theorem~\ref{thm:bip} and Theorem~\ref{thm:counter} accessible.

\section{The deletion recurrence and blow-ups}\label{sec:prelim}

For a graph $G$ of order $n$, set
\[
A(G)\;=\;n!\,\lvert\Aut(G)\rvert\,\Lk(G).
\]
By \cite[Thm.~16]{BMS14}, unwinding the process along all vertex orderings,
\begin{equation}\label{eq:ordersum}
A(G)\;=\;\sum_{\sigma\in \mathfrak{S}_n}\;\prod_{i=1}^{n}\bin{i-1}{d_i(\sigma)}^{-1},
\end{equation}
where $d_i(\sigma)$ is the number of neighbours of $\sigma(i)$ among $\sigma(1),\dots,\sigma(i-1)$. In particular $A(G)=A(\overline G)$ and $A(K_n)=A(\bar K_n)=n!$.

\begin{proposition}[\cite{MW}]\label{prop:deletion}
$A(K_1)=1$, and for every graph $G$ of order $n\ge 2$,
\begin{equation}\label{eq:deletion}
A(G)\;=\;\sum_{v\in V(G)}\frac{A(G-v)}{\bin{n-1}{\deg_G v}} .
\end{equation}
\end{proposition}

\begin{proof}
Partition the sum \eqref{eq:ordersum} according to $\sigma(n)=v$. For fixed $v$, the first $n-1$ factors run over all orderings of $V(G)\setminus\{v\}$ and reproduce $A(G-v)$, while the last factor is $\bin{n-1}{\deg_Gv}^{-1}$.
\end{proof}

Three consequences deserve to be isolated.

\begin{corollary}\label{cor:algo}
Fix a labelling of $V(G)$. Put $F(\{v\})=1$ and, for $|S|\ge 2$,
\[
F(S)=\sum_{v\in S}\frac{F(S\setminus\{v\})}{\bin{|S|-1}{\deg_{G[S]}v}}.
\]
Then
\[
\Lk(G)=\frac{F(V(G))}{n!\,\lvert\Aut(G)\rvert}.
\]
Hence $\Lk(G)$ is computable exactly in $O(n\,2^n)$ arithmetic operations and $O(2^n)$ space.
\end{corollary}

This replaces the enumeration of the $n!/\lvert\Aut(G)\rvert$ construction paths used in \cite{BMS14}. Our implementation combines graph processing in the Wolfram Language with a compiled C helper for the main exact computation. The helper evaluates an integer-scaled form of the subset recurrence modulo several $64$-bit primes, after which the exact integer is reconstructed by the Chinese remainder theorem. At the time of writing, this implementation has been used to compute exact likelihoods for every graph represented in the development version of \texttt{GraphData} \cite{GraphData} through order $29$, as well as selected larger graphs whose structure permits effective preprocessing. In particular, it returns the value in Theorem~\ref{thm:counter} for the circulant $C_{15}(1,4,6)$. A separate C plugin integrated with \texttt{geng} makes the exhaustive search through order $12$ feasible. The complexity of $\Lk$ remains open; see Problem~\ref{prob:complexity}.

\begin{corollary}\label{cor:deck}
$A$ is determined by the vertex-deleted deck $\{G-v: v\in V(G)\}$, counted with multiplicity. In particular, hypomorphic graphs have equal $A$, and $\Lk$ distinguishes them only through the order of their automorphism groups.
\end{corollary}

\begin{proof}
The conclusion is immediate for graphs of order at most $2$. For graphs of order at least $3$, the right-hand side of \eqref{eq:deletion} depends only on the multiset of isomorphism types $G-v$ together with the numbers $\deg_Gv=e(G)-e(G-v)$, and $e(G)$ is determined by the deck by Kelly's lemma.
\end{proof}

The third consequence is the one we shall use. Let $H$ be a graph with $V(H)=\{1,\dots,p\}$ and let $m=(m_1,\dots,m_p)\in\mathbb{Z}_{\ge0}^p$. The \emph{blow-up} $H[m]$ is obtained by replacing vertex $i$ by an independent set $V_i$ of size $m_i$ and joining $V_i$ to $V_j$ completely whenever $ij\in E(H)$. Deleting a vertex of $H[m]$ yields $H[m-e_i]$ for some $i$, so blow-ups of a fixed base form a family closed under vertex deletion and \eqref{eq:deletion} closes on it.

\begin{proposition}\label{prop:blowup}
With $n=\sum_im_i$ and $A(H[m])=1$ when $n=1$,
\begin{equation}\label{eq:blowup}
A(H[m])\;=\;\sum_{i\,:\,m_i>0}\;\frac{m_i\,A(H[m-e_i])}{\bin{n-1}{\sum_{j\sim i}m_j}} .
\end{equation}
\end{proposition}

\begin{proof}
Every vertex of $V_i$ has degree $\sum_{j\sim i}m_j$ and deleting it gives $H[m-e_i]$; apply \eqref{eq:deletion} and group the $n$ terms into the $p$ classes $V_1,\dots,V_p$.
\end{proof}

Taking $H=K_p$ gives all complete multipartite graphs, indexed by partitions of $n$; taking $H=K_2$ gives all complete bipartite graphs; taking $H=C_5$ gives the family of Theorem~\ref{thm:counter}. For the automorphism groups of blow-ups we use the theorem of Sabidussi \cite{Sab59,Sab61} on lexicographic products, in the form given by Grech and Kisielewicz \cite[\S1]{GK22}. For $m>1$, if no two distinct vertices of $H$ have the same open neighbourhood, then $\Aut(H[\bar K_m])\cong\mathfrak{S}_m\wr\Aut(H)$, the natural imprimitive wreath product, and hence has order $\lvert\Aut(H)\rvert\,(m!)^p$.

\section{Complete bipartite graphs}\label{sec:bip}

Specialising \eqref{eq:blowup} to $H=K_2$ and normalising, put
\[
b^{(n)}_a\;=\;\frac{A(K_{a,n-a})}{a!\,(n-a)!},\qquad 0\le a\le n,
\]
so that, since $\lvert\Aut(K_{a,n-a})\rvert=a!\,(n-a)!\,(1+[\,2a=n\,])$,
\begin{equation}\label{eq:Lb}
\Lk(K_{a,n-a})=\frac{b^{(n)}_a}{n!\,\big(1+[\,2a=n\,]\big)} .
\end{equation}
Proposition~\ref{prop:blowup} becomes, using $\bin{n-1}{n-a-1}=\bin{n-1}{a}$,
\begin{equation}\label{eq:brec}
b^{(n)}_a=\frac{b^{(n-1)}_{a-1}}{\bin{n-1}{a-1}}+\frac{b^{(n-1)}_{a}}{\bin{n-1}{a}}
\qquad (1\le a\le n-1),
\end{equation}
with $b^{(n)}_0=b^{(n)}_n=1$ for all $n\ge1$ and $b^{(1)}_0=b^{(1)}_1=1$. Note $b^{(n)}_a=b^{(n)}_{n-a}$.

\begin{lemma}\label{lem:b1}
For $n\ge 2$ we have $b^{(n)}_1=b^{(n)}_{n-1}=\dfrac{1}{(n-1)!}\displaystyle\sum_{k=0}^{n-1}k!$, and $1<b^{(n)}_1\le 2$ with equality if and only if $n\in\{2,3\}$.
\end{lemma}

\begin{proof}
$K_{1,n-1}$ is the star, and $\Lk(K_{1,n-1})=n(n!)^{-2}\sum_{k=0}^{n-1}k!$ by \cite[Prop.~20]{BMS14}; substituting into \eqref{eq:Lb} with $\lvert\Aut\rvert=(n-1)!$ gives the formula. For the bounds, isolate the top term and observe that
$\sum_{k=0}^{n-2}k!\le (n-2)!+\sum_{k=0}^{n-3}k!\le (n-2)!+(n-2)(n-3)!=2(n-2)!$,
which is smaller than $(n-1)!=(n-1)(n-2)!$ as soon as $n\ge4$. Hence $b^{(n)}_1=1+\sum_{k=0}^{n-2}k!/(n-1)!<2$ for $n\ge4$, while $b^{(2)}_1=b^{(3)}_1=2$; and $b^{(n)}_1>1$ because the term $k=n-1$ alone contributes $1$.
\end{proof}

\begin{lemma}\label{lem:bbound}
For all $n\ge1$ and all $a$ we have $b^{(n)}_a\le 2$. Moreover $b^{(n)}_a\le 4/(n-1)$ whenever $n\ge4$ and $2\le a\le n-2$.
\end{lemma}

\begin{proof}
Induction on $n$. For $n\le 3$ every value is $1$ or $2$ by Lemma~\ref{lem:b1}. Let $n\ge4$ and assume the first claim for $n-1$. If $a\in\{0,n\}$ then $b^{(n)}_a=1$, and if $a\in\{1,n-1\}$ then $b^{(n)}_a\le2$ by Lemma~\ref{lem:b1}. If $2\le a\le n-2$ then $1\le a-1$ and $a\le n-2$, so $\bin{n-1}{a-1}\ge n-1$ and $\bin{n-1}{a}\ge n-1$, and \eqref{eq:brec} gives $b^{(n)}_a\le 2/(n-1)+2/(n-1)=4/(n-1)\le 2$.
\end{proof}

\begin{lemma}\label{lem:diff}
For $1\le a\le n-2$,
\[
b^{(n)}_a-b^{(n)}_{a+1}=\frac{b^{(n-1)}_{a-1}}{\bin{n-1}{a-1}}-\frac{b^{(n-1)}_{a+1}}{\bin{n-1}{a+1}} .
\]
\end{lemma}

\begin{proof}
Subtract the instance $a+1$ of \eqref{eq:brec} from the instance $a$; the two middle terms $b^{(n-1)}_{a}/\bin{n-1}{a}$ cancel.
\end{proof}

\begin{proposition}\label{prop:mono}
For every $n\ge 4$ the sequence $b^{(n)}_1,b^{(n)}_2,\dots,b^{(n)}_{\lfloor n/2\rfloor}$ is strictly decreasing.
\end{proposition}

\begin{proof}
Induction on $n$. For $n=4$ we have $b_1=5/3>b_2=4/3$, and for $n=5$, $b_1=17/12>b_2=23/36$. Let $n\ge 6$, assume the statement for $n-1$, and fix $a$ with $1\le a\le \lfloor n/2\rfloor-1$. By Lemma~\ref{lem:diff} it suffices to prove
\begin{equation}\label{eq:target}
\frac{b^{(n-1)}_{a-1}}{\bin{n-1}{a-1}}>\frac{b^{(n-1)}_{a+1}}{\bin{n-1}{a+1}} .
\end{equation}

Assume first $a\ge2$, so that $1\le a-1<a+1\le\lfloor n/2\rfloor$. If $a+1\le\lfloor (n-1)/2\rfloor$ then the inductive hypothesis gives $b^{(n-1)}_{a-1}>b^{(n-1)}_{a+1}$, while $a+1\le\lfloor (n-1)/2\rfloor$ forces $\bin{n-1}{a-1}<\bin{n-1}{a+1}$, and \eqref{eq:target} follows. Otherwise $a+1>\lfloor (n-1)/2\rfloor$ together with $a+1\le\lfloor n/2\rfloor$ forces $n$ even and $a+1=n/2$, so $n-1=2\mu+1$ with $\mu=n/2-1=a$. Then $b^{(n-1)}_{a+1}=b^{(n-1)}_{n-1-(a+1)}=b^{(n-1)}_{a}$ and $\bin{n-1}{a+1}=\bin{2\mu+1}{\mu+1}=\bin{2\mu+1}{\mu}=\bin{n-1}{a}$, so \eqref{eq:target} reduces to $b^{(n-1)}_{a-1}/\bin{n-1}{a-1}>b^{(n-1)}_{a}/\bin{n-1}{a}$, which holds because $1\le a-1<a\le \lfloor (n-1)/2\rfloor$ gives $b^{(n-1)}_{a-1}>b^{(n-1)}_{a}$ by induction and $\bin{n-1}{a-1}<\bin{n-1}{a}$.

Now let $a=1$. Then $b^{(n-1)}_{0}=1$ and $\bin{n-1}{0}=1$, so \eqref{eq:target} reads $b^{(n-1)}_{2}<\bin{n-1}{2}$. Since $n\ge6$, Lemma~\ref{lem:bbound} gives $b^{(n-1)}_2\le 4/(n-2)\le1<\bin{n-1}{2}$.
\end{proof}

\begin{proof}[Proof of Theorem~\ref{thm:bip}]
By \eqref{eq:Lb} and Proposition~\ref{prop:mono}, $\Lk(K_{a,n-a})$ is strictly decreasing for $1\le a\le\lfloor n/2\rfloor-1$, and the passage from $a=\lfloor n/2\rfloor -1$ to $a=\lfloor n/2\rfloor$ decreases it as well, strictly, since for $n$ even the additional factor $(1+[2a=n])^{-1}=1/2$ only helps. It remains to compare with $a=0$, where $\Lk(\bar K_n)=1/n!$. For $n\ge 6$ we have $2\le\lfloor n/2\rfloor\le n-2$, so Lemma~\ref{lem:bbound} gives $b^{(n)}_{\lfloor n/2\rfloor}\le 4/(n-1)<1$; for $n=4$, $b_2^{(4)}/2=2/3<1$ and for $n=5$, $b^{(5)}_2=23/36<1$.
\end{proof}

\begin{remark}
The function $a\mapsto\Lk(K_{a,n-a})$ is not monotone on $0\le a\le\lfloor n/2\rfloor$; the star $a=1$ is the maximum of the family and lies above the empty graph $a=0$. For $n=10$ the values are $2.76\times 10^{-7}$, $3.11\times 10^{-7}$, $3.63\times 10^{-8}$, $1.20\times 10^{-9}$, $2.29\times 10^{-11}$, $5.27\times 10^{-13}$.
\end{remark}

\section{Exhaustive verification for \texorpdfstring{$n\le 12$}{n <= 12}}\label{sec:small}

The structured searches below use the following elementary screening, which
combines the deletion recurrence, the lower bound in \eqref{eq:BMSbounds},
and the fact that automorphisms preserve degrees. Write
$P_n=\prod_{i=1}^{n}\bin{i-1}{\lfloor (i-1)/2\rfloor}$ and let $m_d(G)$ be
the number of vertices of $G$ of degree $d$.

\begin{lemma}\label{lem:screen}
If $n\ge2$ and $\Lk(G)\le\lambda$, then
\[
\prod_{d\ge0}m_d(G)!\;\ge\;\lvert\Aut(G)\rvert\;\ge\;
\frac{1}{n\lambda P_{n-1}}
\sum_{v\in V(G)}\bin{n-1}{\deg_Gv}^{-1}
\;\ge\;\frac{1}{\lambda P_n}.
\]
\end{lemma}

\begin{proof}
The left inequality holds because $\Aut(G)$ embeds in the direct product of
the symmetric groups on the degree classes.  Applying \eqref{eq:BMSbounds}
to each graph $G-v$ gives
\[
A(G-v)=(n-1)!\lvert\Aut(G-v)\rvert\Lk(G-v)\ge\frac{(n-1)!}{P_{n-1}}.
\]
Substitution in \eqref{eq:deletion}, followed by
$\Lk(G)=A(G)/(n!\lvert\Aut(G)\rvert)\le\lambda$, gives the middle inequality.  The
last follows from
$\bin{n-1}{\deg_Gv}\le\bin{n-1}{\lfloor(n-1)/2\rfloor}$ and the definition
of $P_n$.
\end{proof}

For the exhaustive proof we use one implementation uniformly through order
$12$.  Setting
$D_i=\operatorname{lcm}\{\bin{i-1}{k}:0\le k\le i-1\}$ and
$\widehat F(S)=F(S)\prod_{i\le |S|}D_i$, the recurrence of
Corollary~\ref{cor:algo} becomes
\[
\widehat F(S)=\sum_{v\in S}\widehat F(S\setminus\{v\})
\frac{D_{|S|}}{\displaystyle\bin{|S|-1}{\deg_{G[S]}v}},
\]
in which every factor is a positive integer.  For $n\le12$,
$\widehat F(V)\le n!\prod_{i\le n}D_i<2^{90}$.  The integrated
\texttt{geng} plugin uses checked unsigned $128$-bit arithmetic for this
recurrence, exact \texttt{nauty} automorphism orders, and exact $192$-bit
cross-products for likelihood comparisons.  It rejects larger orders.  It
enumerates every unlabelled graph through order $11$ and, using complement
symmetry, one representative from every complement orbit at order $12$.  As
a global check, it reproduces the extrema previously known through order $10$.

At order $11$, the program streams all $1\,018\,997\,864$ unlabelled graphs
and obtains
\[
\min_G\Lk(G)=\frac{20692621}{2012184939663360000000}.
\]
The two minimisers are $K_{5,6}$ and its complement. Graph6 representations
of the complement and $K_{5,6}$ are \verb|J~{?GKF@wN_| and
\verb|J??F~z{~Fw?|, respectively.

For $n=12$, complement symmetry reduces the search to one representative
from each complement orbit. There are exactly
$82\,545\,586\,656=(165\,091\,172\,592+720)/2$ such orbits, since there are
$165\,091\,172\,592$ unlabelled graphs of order $12$, of which $720$ are
self-complementary. The search obtains
\[
\min_G\Lk(G)=\frac{20692621}{11155553305493667840000000}.
\]
The two minimisers are $K_{6,6}$ and its complement, with graph6
representations \verb+KsaCB|}^b{No+ and \verb+K~~w?CB?wF_^+. This proves
Theorem~\ref{thm:small}.

\begin{remark}
The minimum is well separated. For $n=10$ the second smallest value is $8.39\times 10^{-12}$, larger than the minimum by a factor of about $16$.
\end{remark}

\section{Exact minima at \texorpdfstring{$n=13$ and $n=14$}{n = 13 and n = 14}}\label{sec:min13-14}

The lower bound in \eqref{eq:BMSbounds} makes enumeration by automorphism
group possible without generating every graph. To meet or beat the balanced
complete bipartite graph at orders $13$ and $14$, respectively, a graph must
have at least $183$ and $2552$ automorphisms. The unrounded bounds are
$182.255\ldots$ and $2551.577\ldots$, so equality candidates are included.

We certify these minima by canonical signed twin quotients.
The maximal twin classes of a graph form a canonical partition: each
nonsingleton class induces an independent set or a complete graph, and
its vertices have the same neighbours outside the class. Collapsing these
classes gives a signed, vertex-weighted \emph{twin quotient}; its weights
sum to $n$, and its signs record the independent or complete classes.
The twin-free case has quotient order $q=n$. Complementation reverses
all nonsingleton signs and preserves likelihood, so one representative
of each complement pair suffices.

For every signed weight family, exact lower bounds give a necessary
automorphism order for a graph to meet or beat the incumbent. Families
whose maximum possible automorphism order is too small are excluded.
For $q\le9$, all eligible unlabelled bases and compatible colourings
are enumerated. For $q\ge10$, the pinned complete subgroup catalogue
supplies the orbital cover: an invariant graph is a union of orbits on
unordered pairs, with all compatible colours and fixed-point adjacencies
included. At orders $13$ through $16$, every viable family in this range
requires a nontrivial quotient automorphism group. The full
colour-preserving automorphism group of a possible competitor therefore
occurs in the cover. The covering argument is detailed in the proof of
Proposition~\ref{prop:min18} below.

Each expansion must have exactly the prescribed canonical twin partition
and meet its automorphism threshold. The exact pruning in the computational
verification section retains all possible improvements and equalities.
Every surviving representation is evaluated exactly, all tied minima are
logged, and canonical labelling and complementation recover the complete
set of minimizing isomorphism classes. The fixed-plan audits check every
family, task identity and range, survivor count and minimum log;
independent rational subset recurrences verify the final witnesses.

At orders $13$ and $14$, the complete searches cover $126$ and $67$ viable
families in $165$ and $50$ tasks, respectively. They require only $20$
and $9$ exact survivor evaluations. In each case the complete minimizer
set consists of the balanced complete bipartite graph and its complement.
Together with Theorem~\ref{thm:small}, these computations prove
Theorem~\ref{thm:firstfailure}.

At the degree-sequence level, the final inequality of Lemma~\ref{lem:screen}
leaves $608\,239$ of the $836\,315$ graphical sequences at order $13$ and
$1\,041\,046$ of the $3\,166\,852$ sequences at order $14$.  The strengthened
middle inequality reduces these counts to $152\,691$ and $175\,878$.
Applying the deletion recurrence once more and minimizing over every possible
neighbour-degree assignment, a relaxation that remains a rigorous lower
bound, leaves $69\,797$ and $75\,888$ sequences, or $34\,949$ and $37\,944$
orbits under complementation.  Thus degree information alone substantially
narrows the search but does not settle either order.

\section{The first counterexample at \texorpdfstring{$n=15$}{n = 15}}\label{sec:counter}

\begin{proof}[Exact evaluation of the candidate]
Let $H=C_5$ with $V(H)=\mathbb{Z}_5$ and $m=(3,3,3,3,3)$, so $G=C_5[\bar K_3]$ is a $6$-regular graph of order $15$ with $45$ edges; it is isomorphic to the circulant $C_{15}(1,4,6)$, the parts being the cosets of the subgroup of order $3$ in $\mathbb{Z}_{15}$. No two distinct vertices of $C_5$ have the same open neighbourhood, so $\lvert\Aut(G)\rvert=\lvert\Aut(C_5)\rvert\,(3!)^5=10\times 6^5=77760$ by the theorem of Sabidussi. Iterating \eqref{eq:blowup} over the $4^5$ vectors $m'\le m$ in exact rational arithmetic gives
\[
A(G)=\frac{63977511069907}{427822618422142944000000},
\]
and $\Lk(G)=A(G)/(15!\times 77760)$ is the stated value. Independently, Corollary~\ref{cor:algo} applied to $G$ with $2^{15}$ subsets and rational arithmetic returns the same $A(G)$. Theorem~\ref{thm:bip} or a direct application of \eqref{eq:brec} gives $\Lk(K_{7,8})$, and the comparison is a rational inequality.
\end{proof}

At order $15$, put
\[
\lambda=
\frac{63977511069907}{43503039787261205233506826321920000000000}.
\]
Lemma~\ref{lem:screen} requires at least $483$ automorphisms for any graph
with likelihood at most $\lambda$. The family bounds and complete covers
leave $237$ viable families, including the twin-free layer, in $508$ tasks.
Exact evaluation of the $113$ surviving representations finds no value
below $\lambda$. The complete tied-minimum audit gives precisely
$C_5[\bar K_3]$ and its complement. These are nonisomorphic, with
$45$ and $60$ edges, respectively. The latter is the Catlin graph $C_5[K_3]$,
the triangle blow-up of $C_5$. This proves Theorem~\ref{thm:counter}.

\section{Exact minima at orders \texorpdfstring{$16$ through $18$}{16 through 18}}\label{sec:min16-18}

\begin{proposition}\label{prop:min16}
Among all simple graphs of order $16$, the likelihood is minimised exactly
by $K_{8,8}$ and its complement.
\end{proposition}

\begin{proof}
The candidate value is
\[
\Lk(K_{8,8})=
\frac{7628328998218493}{107497751435913928580204708114517196800000000000}.
\]
The lower bound in \eqref{eq:BMSbounds} forces every graph meeting or
beating this value to have at least $1556$ automorphisms; the unrounded
threshold lies strictly between $1555$ and $1556$, so equality candidates
are included. The signed-family reduction above leaves $276$ viable
families, including the twin-free layer. All $549$ fixed-plan tasks
complete, and the $113$ exact survivor evaluations attain no smaller
value. Auditing every tied-minimum log, then canonically labelling the
minimizers and their complements, gives exactly the two stated classes.
Their likelihoods are independently recomputed by the rational subset
recurrence. Compact plans, coverage certificates and the complete
minimizer sets are retained in the computational records; measured search times
are in Table~\ref{tab:fullreruns}.
\end{proof}

\paragraph{Completed order-$17$ campaign.}
Let $H_{17}$ be obtained from a five-cycle by replacing its vertices, in
cyclic order, by
\[
K_4,\quad\bar K_4,\quad\bar K_4,\quad K_4,\quad K_1,
\]
joining consecutive classes completely and adding no edges between
nonconsecutive classes. This mixed blow-up is $8$-regular with $68$ edges.
Its maximal twin classes have sizes $4,4,4,4,1$. The signed weighted
quotient has two automorphisms, giving
$\lvert\Aut(H_{17})\rvert=2(4!)^4=663\,552$.
It is self-complementary: a permutation of the five classes maps its
signed quotient to the complemented quotient and extends to a graph
isomorphism.

\begin{proposition}\label{prop:min17}
Among all simple graphs of order $17$, the likelihood is minimised by
exactly one isomorphism class, represented by $H_{17}$. Its value is
\[
\Lk(H_{17})=
\frac{46676188724227801}{671983801047531669544656802525006213939200000000000}
=6.9460\ldots\times10^{-35}.
\]
In particular,
\[
\frac{\Lk(H_{17})}{\Lk(K_{8,9})}
=\frac{457426649497432449800}{2200453451294991512083}
=0.207878\ldots<1.
\]
\end{proposition}

\begin{proof}
The exact signed-family bounds exclude $591$ of the $1334$ families with
quotient order below $17$. The remaining $743$ families, together with the
twin-free layer, are covered by all eligible unlabelled bases and compatible
colourings for $q\le9$, and complete subgroup, colour-palette and orbital
covers for $q\ge10$. In the latter range every necessary quotient-group
order exceeds one, so the pinned subgroup catalogue covers the full
colour-preserving automorphism group of every possible competitor. Fixed
points are restored with all compatible colours and adjacency choices.
The covering argument is detailed for order $18$ below. Recomputing all
family bounds and cover sizes and checking every fixed record interval
gives $447\,803\,939\,548$ higher-$q$ representations and $479\,982$ queue
tasks, with no gap in any assigned range.

The initial incumbent was the exactly evaluated independent $C_5$ blow-up
with part sizes $(3,3,3,4,4)$. Every exclusion uses an exact necessary
condition or a lower bound strictly greater than that incumbent, retaining
all improvements and equalities. Every surviving graph is evaluated, and
every new running minimum and tie is logged. The full audit checks every
committed survivor-file hash, evaluation count and minimum-log hash,
reconstructs every block and task minimum, and verifies that every cursor
reaches its prescribed endpoint. Canonical labelling of every final
minimum and its complement leaves exactly the single class $H_{17}$.
An independent rational subset deletion recurrence, with the automorphism
order obtained above, recomputes the displayed likelihood for both the
witness and its complement. Thus no lower value or additional minimising
class occurs within the exhaustive cover.
\end{proof}

\paragraph{Completed order-$18$ campaign.}
The improved signed twin-family search also gives the following result.

\begin{proposition}\label{prop:min18}
Among all simple graphs of order $18$, the likelihood is minimised exactly
by $K_{9,9}$ and its complement, with value
\[
\Lk(K_{9,9})=
\frac{2200453451294991512083}{2881657382291313298001639051759143747136048332800000000000000}.
\]
\end{proposition}

\begin{proof}
Partition a graph into its canonical twin classes, retaining each class's
size and internal complete or independent sign, modulo complementation.
The exact signed-family lower bounds described under Computational
verification exclude $1066$ of the $1894$ families with quotient order below
$18$: their necessary full automorphism orders exceed their factorial
upper bounds. There remain $828$ signed families and the twin-free layer.
For each retained family the automorphism group has the internal factorial
product times the colour-preserving quotient automorphism order.

For quotient order $q\le9$, every eligible unlabelled base and every compatible
colouring is enumerated. For $q\ge10$, all necessary quotient-group orders
exceed one. The complete subgroup catalogue for degrees through $18$
\cite{SubgroupData,Holt} therefore covers the full colour-preserving
automorphism group of every possible competitor. Remove global fixed
points from its action and then restore them. Every support orbit receives
every compatible class colour; an orbit that forces twins cannot receive a
singleton colour. All unions of pair orbitals are enumerated. On the fixed
points, one may either use all labelled orbital choices with a sorted colour
palette, or all unlabelled graphs with every compatible colour palette.
Permuting fixed points centralises the support action, so each construction
covers every required coloured quotient up to relabelling.

The complete higher-$q$ cover has $10\,937\,184\,984$ representations,
including subgroup duplicates. The fixed-plan audit checks all $45\,148$
shard identities and an exact partition of every complete record interval.
Canonical twin and automorphism checks precede graphwise exact pruning,
which discards only strict excess over the incumbent. Every survivor is
evaluated exactly. Every new running minimum and every tie is logged;
hence every graph attaining the final global minimum occurs in its block's
retained minimum log. Auditing all committed survivor files and minimum
logs, then canonically labelling all final minima and their complements,
leaves precisely two isomorphism classes. Their structures are $K_{9,9}$
and its complement. An independent rational subset recurrence, with the
analytic automorphism order $2(9!)^2$, gives the displayed value for both.
Thus no smaller value or further minimising class occurs.
\end{proof}

The finite subgroup catalogue is an input to this computation. Its pinned
archive hashes, complete population checks against the known $7\,274\,651$
classes of $\mathfrak S_{18}$, exact group orders, family bounds and full
range plan are retained in the computational records. The graph search and its audits
do not claim a new independent classification of those subgroups.

\section{Further counterexamples among blow-ups of \texorpdfstring{$C_5$}{C5}}\label{sec:further}

The advantage of the $C_5$ blow-ups over the balanced complete bipartite
graph grows very quickly with $n$; for $n=25$,
\[
\Lk\big(C_5[\bar K_5]\big)=1.202\ldots\times 10^{-77}
\qquad\text{against}\qquad
\Lk(K_{12,13})=6.960\ldots\times 10^{-72},
\]
a factor of about $5.8\times 10^{5}$.  Exact enumeration of every positive
part vector, modulo the dihedral symmetries of $C_5$, shows that through
$n=50$ the minimum among positive $C_5$ part vectors is always attained by a
nearly balanced vector, meaning one whose coordinates differ by at most one.
It is smaller than $\Kbal$ at $n=15$, at $n=17$, and at
every order from $19$ through $50$.  Exact evaluation of the best nearly
balanced vectors continues this sequence of counterexamples without a gap
through $n=100$; selected early examples are listed in
Table~\ref{tab:c5more}.

\begin{table}[htbp]
\centering
\begin{tabular}{@{}rcr@{.}l@{}}
\toprule
$n$ & $m$ & \multicolumn{2}{c}{$\Lk(C_5[m])/\Lk(\Kbal)$}\\
\midrule
$17$ & $(3,3,3,4,4)$ & $0$ & $32136\ldots$\\
$19$ & $(3,4,4,4,4)$ & $0$ & $072155\ldots$\\
$20$ & $(4,4,4,4,4)$ & $0$ & $088189\ldots$\\
$21$ & $(4,4,4,4,5)$ & $0$ & $018452\ldots$\\
$22$ & $(4,4,4,5,5)$ & $0$ & $26319\ldots$\\
$23$ & $(4,5,4,5,5)$ & $0$ & $0040481\ldots$\\
$24$ & $(4,5,5,5,5)$ & $0$ & $060271\ldots$\\
$25$ & $(5,5,5,5,5)$ & $1$ & $7271\ldots\times 10^{-6}$\\
\bottomrule
\end{tabular}
\caption{Selected further counterexamples among blow-ups of $C_5$. The
displayed ratios are rounded from exact rational computations.}
\label{tab:c5more}
\end{table}

For the positive vectors in Table~\ref{tab:c5more}, the five parts are the
twin classes, so $\lvert\Aut(C_5[m])\rvert=(\prod_{i=1}^5m_i!)
\lvert\operatorname{Stab}_{\Aut(C_5)}(m)\rvert$; together with \eqref{eq:blowup}, this
gives every ratio in the table exactly.
Within the $C_5[m]$ family, $\Kbal$ is again smaller at $n=16$ and $n=18$,
where the five parts cannot all have equal size; the failure of the conjecture
is therefore not monotone in $n$ within this family. The completed campaigns above certify $K_{8,8}$ and $K_{9,9}$,
respectively, together with their complements, as the only global minimisers
at those two orders.

\section{Asymptotics}\label{sec:asym}

We first evaluate $\Lk(K_{m,m})$. Iterating \eqref{eq:brec} downwards from $(m,m)$ expresses $b^{(2m)}_m$ as a sum over lattice paths. Formally, call a \emph{deletion path} from $(a,b)$ a sequence of steps each decreasing one coordinate by $1$, terminating when a coordinate reaches $0$, and give the step that decreases the first coordinate at state $(a',b')$ the cost $\log_2\bin{a'+b'-1}{b'}$, the other step the cost $\log_2\bin{a'+b'-1}{a'}$. Unwinding \eqref{eq:brec} and using $b^{(k)}_0=1$,
\begin{equation}\label{eq:pathsum}
b^{(2m)}_m=\sum_{\pi}2^{-\mathrm{cost}(\pi)},
\end{equation}
the sum being over all deletion paths from $(m,m)$. Let $W(m)=\min_\pi \mathrm{cost}(\pi)$.

\begin{lemma}\label{lem:greedy}
$W(m)=\sum_{j=0}^{m-1}\log_2\bin{m+j}{j}$, attained by the path that empties one side first.
\end{lemma}

\begin{proof}
Let $f(a,b)$ denote the minimum cost of a deletion path from $(a,b)$, with $f(a,0)=f(0,b)=0$. We show by induction on $a+b$ that for $a\le b$,
\[
f(a,b)=S(a,b):=\sum_{j=0}^{a-1}\log_2\bin{b+j}{j},
\]
which is the cost of emptying the first side; by symmetry this suffices. The cases $a=0$ and $a+b\le2$ are clear. Let $1\le a\le b$. Emptying the first side has cost $\log_2\bin{a+b-1}{b}+S(a-1,b)=S(a,b)$, since $\bin{a+b-1}{b}=\bin{a+b-1}{a-1}$. So $f(a,b)\le S(a,b)$, and it remains to exclude the other first step, that is to show
\begin{equation}\label{eq:exchange}
S(a,b)\;\le\; \log_2\bin{a+b-1}{a}+f(a,b-1).
\end{equation}
Suppose first $a\le b-1$, so that $f(a,b-1)=S(a,b-1)$ by induction. Since
$\bin{b+j}{j}\big/\bin{b-1+j}{j}=(b+j)/b$, we get
\begin{align*}
S(a,b)-S(a,b-1)
&=\sum_{j=0}^{a-1}\log_2\frac{b+j}{b}\\
&=\log_2\frac{\prod_{j=0}^{a-1}(b+j)}{b^{a}},\\
\log_2\bin{a+b-1}{a}
&=\log_2\frac{\prod_{j=0}^{a-1}(b+j)}{a!},
\end{align*}
so \eqref{eq:exchange} is equivalent to $a!\le b^{a}$, which holds because $a\le b$. If instead $a=b$, then $f(a,b-1)=S(a-1,a)$ by induction and
$S(a,a)-S(a-1,a)=\log_2\bin{2a-1}{a-1}=\log_2\bin{2a-1}{a}$, so \eqref{eq:exchange} holds with equality, as symmetry demands.
\end{proof}

\begin{lemma}\label{lem:W}
$W(m)=\big(2-\tfrac{1}{2\ln 2}\big)m^2+O(m\log m)$.
\end{lemma}

\begin{proof}
By Lemma~\ref{lem:greedy}, $W(m)\ln 2=\sum_{j=0}^{m-1}\big(\ln (m+j)!-\ln m!-\ln j!\big)$. Stirling gives $\ln k!=k\ln k-k+O(\log k)$, and comparing the three sums with the integrals $\int x\ln x\,dx$ yields
\[
\sum_{j=0}^{m-1}\ln(m+j)!=\tfrac32m^2\ln m+(2\ln 2-\tfrac94)m^2+O(m\log m),
\]
\begin{align*}
m\ln m!&=m^2\ln m-m^2+O(m\log m),\\
\sum_{j=0}^{m-1}\ln j!&=\tfrac12m^2\ln m-\tfrac34m^2+O(m\log m).
\end{align*}
The $m^2\ln m$ terms cancel and the remainder is $(2\ln2-\tfrac12)m^2+O(m\log m)$.
\end{proof}

\begin{proof}[Proof of Theorem~\ref{thm:asym}]
There are at most $\bin{2m}{m}\le4^m$ deletion paths from $(m,m)$, so
\eqref{eq:pathsum} gives
\begin{align*}
2^{-W(m)}&\le b^{(2m)}_m\le 4^m2^{-W(m)},\\
\log_2(1/b^{(2m)}_m)&=W(m)+O(m).
\end{align*}
By \eqref{eq:Lb},
\[
\log_2\frac{1}{\Lk(K_{m,m})}
=\log_2\frac{1}{b^{(2m)}_m}+\log_2(2\times(2m)!).
\]
Lemma~\ref{lem:W} with $n=2m$ and $m^2=n^2/4$ therefore gives
\[
\log_2\frac1{\Lk(K_{m,m})}=\Big(\tfrac12-\tfrac1{8\ln2}\Big)n^2+O(n\log n).
\]
For odd $n$ the same argument applied to $(m,m+1)$ changes only the $O(n\log n)$ term.

For the second statement, let $g(n)$ be the number of isomorphism classes of graphs of order $n$. Since $\Lk$ is a probability distribution on these classes, $\min_G\Lk(G)\le 1/g(n)$, and $g(n)\ge 2^{\bin n2}/n!$, so $\log_2(1/\min_G\Lk)\ge\bin n2-\log_2n!=\tfrac{n^2}2+O(n\log n)$. Conversely \eqref{eq:BMSbounds} with $\lvert\Aut(G)\rvert\le n!$ gives $\min_G\Lk\ge 1/(n!P_n)$, and $\log_2P_n=\sum_{i=1}^n\log_2\bin{i-1}{\lfloor (i-1)/2\rfloor}=\bin n2+O(n\log n)$, whence $\log_2(1/\min_G\Lk)\le \tfrac{n^2}2+O(n\log n)$.

Finally $\tfrac12-\tfrac1{8\ln2}<\tfrac12$ strictly, so the two exponents differ in the leading term and the ratio is $2^{(1/(8\ln2)+o(1))n^2}$; in particular $\Lk(\Kbal)>\min_G\Lk(G)$ for all large $n$.
\end{proof}

\begin{remark}
Numerically $\log_2(1/\Lk(\Kbal))/(n^2/2)$ equals $0.8158$, $0.7515$, $0.7130$, $0.6852$ and $0.6554$ for $n=10,40,80,160,640$, descending towards $2\times(\tfrac12-\tfrac1{8\ln2})=1-\tfrac1{4\ln2}=0.6393\ldots$, and $W(m)/m^2$ equals $1.2615$ at $m=320$ against the limit $2-\tfrac1{2\ln 2}=1.27865\ldots$.
\end{remark}

\section{The entropy of the process}\label{sec:entropy}

\begin{proof}[Proof of Theorem~\ref{thm:entropy}]
The random choices are independent across steps, so the entropy of the whole sequence of choices is
\[
\mathcal{H}_n=\sum_{t=1}^{n}\left(\log_2 t+
\mathbb{E}_k\log_2\bin{t-1}{k}\right),
\]
with $k$ uniform on $\{0,\dots,t-1\}$. Since
$\log_2\bin Nk=NH_2(k/N)+O(\log N)$ uniformly, where $H_2$ is the
binary entropy in bits, and $\int_0^1H_2(x)\,dx=1/(2\ln2)$, we get
\[
\mathbb{E}_k\log_2\bin{t-1}{k}=\frac{t}{2\ln2}+O(\log t),
\qquad
\mathcal H_n=\frac{n^2}{4\ln2}+O(n\log n).
\]
The isomorphism class is a function of the choices, so its entropy is at most
$\mathcal H_n$; conversely, conditioned on the class, the number of choice
sequences producing it is at most $n!$, so the conditional entropy is at most
$\log_2n!=O(n\log n)$. The comparison with
$\log_2g(n)=\tfrac{n^2}{2}+O(n\log n)$ is as in the proof of
Theorem~\ref{thm:asym}.

It remains to justify the assertion about a typical output. Let $Y_n$ denote
the complete choice history and put
\[
Z_n=-\log_2\Pr(Y_n)
=\sum_{t=1}^{n}\left(\log_2t+
\log_2\bin{t-1}{K_t}\right),
\]
where the $K_t$ are independent and uniform on $\{0,\dots,t-1\}$. The
summands are independent and bounded in absolute value by $O(t)$, so
$\operatorname{Var}(Z_n)=O(n^3)$. Chebyshev's inequality and the preceding
entropy computation give
\[
Z_n=\frac{n^2}{4\ln2}+o_{\mathrm p}(n^2).
\]
Now set $I_n=-\log_2\Lk(G_n)$. Since $G_n$ is a function of $Y_n$,
\[
D_n:=Z_n-I_n=-\log_2\Pr(Y_n\mid G_n)\ge0.
\]
Moreover $\mathbb E D_n=H(Y_n\mid G_n)\le\log_2n!=O(n\log n)$, because
each isomorphism class has at most $n!$ choice histories. Markov's inequality
therefore gives $D_n=o_{\mathrm p}(n^2)$, and hence
$I_n/n^2\xrightarrow{\mathrm p}1/(4\ln2)$.
\end{proof}

Exact values of the entropy for $n\le 9$, computed from all isomorphism
classes, are
\[
0,1,1.918,3.198,4.793,6.841,9.444,12.717,16.729
\]
bits, against
\[
\log_2g(n)=0,1,2,3.46,5.09,7.29,10.03,13.59,18.07.
\]
The ratio decreases slowly from $0.939$ at $n=6$ to $0.925$ at $n=9$,
consistent with the limiting ratio $1/(2\ln 2)=0.7213\ldots$ implied by
Theorem~\ref{thm:entropy}.  The concentration statement implies that there is
a sequence $\varepsilon_n\downarrow0$ such that the typical set
\[
\mathcal T_n=\left\{G:\left|-\frac{1}{n^2}\log_2\Lk(G)
-\frac{1}{4\ln2}\right|\le\varepsilon_n\right\}
\]
satisfies $\Pr(G_n\in\mathcal T_n)=1-o(1)$.  Bounding the likelihood of each
member of $\mathcal T_n$ above and below and summing shows that
$\lvert\mathcal T_n\rvert=2^{n^2/(4\ln2)+o(n^2)}$.  Thus the process
concentrates on a $2^{-(1/2-1/(4\ln2))n^2+o(n^2)}$ fraction of all
isomorphism classes while still spreading over $2^{\Theta(n^2)}$ of them,
which is the precise sense in which the uniform monkey is a poor but not
degenerate sampler.

\section{The maximum, and open problems}\label{sec:remarks}

For the maximum no analogue of Theorem~\ref{thm:bip} is available, and the situation is genuinely irregular. Up to complementation, the maximisers of $\Lk$ are threshold graphs for $n\le 6$; for $n=7$ and $n=8$ they are split but not threshold; and for $n=9$ the maximiser is the disjoint union of two isolated vertices with a caterpillar on seven vertices, which is not split, being obtained from $P_6$ by attaching a pendant vertex at the third vertex. Thus neither the threshold graphs nor the broader class of split graphs contains the maximisers for all orders. The next three exact maximum values are
\[
\max_{\lvert V(G)\rvert=10}\Lk(G)=\frac{219903017}{4938071040000},
\]
\[
\max_{\lvert V(G)\rvert=11}\Lk(G)=\frac{7472883629}{1140694410240000},
\]
\[
\max_{\lvert V(G)\rvert=12}\Lk(G)=\frac{18096017946691}{17566693917696000000}.
\]
The complementary maximising pairs have graph6 representations
\verb|I????GICo| and \verb|Ihuz~~~~w| for $n=10$;
\verb|J?????A@OR?| and \verb|Jhuz~~~~~~_| for $n=11$; and
\verb+K????C?@?PAE+ and \verb+Kh~u||~~~~~~+ for $n=12$.
What does appear stable is the order of magnitude: $-\log_2\max_G\Lk(G)$ equals
\[
0,1,1.585,2.469,3.815,5.487,7.439,9.684,12.021,14.455,17.220,19.889
\]
for $n\le 12$, consistent with $\max_G\Lk(G)=2^{\Theta(n)}/n!$.

\begin{problem}\label{prob:min}
Determine $\arg\min_G\Lk(G)$ for $n\ge19$. Are the minimisers always
blow-ups, and can their base graphs be chosen from a fixed finite family, or
must their orders grow with $n$? Theorem~\ref{thm:asym} shows that
$\lvert\Aut\rvert$ is of secondary importance on the $n^2$ logarithmic scale
and suggests looking for graphs whose typical vertex orderings have back
degrees near the centre of their range.
\end{problem}

\begin{problem}
Prove $A(G)\le C^n\lvert\Aut(G)\rvert$ for an absolute constant $C$. Equivalently,
prove
\[
\max_G\Lk(G)=2^{O(n)}/n!.
\]
\end{problem}

\begin{problem}\label{prob:complexity}
Settle the complexity of $\Lk$, raised in \cite{BMS14} and repeated in \cite{DMS18}. Is exact evaluation of $A$ $\#\mathrm P$-hard? Corollary~\ref{cor:algo} is a dynamic program over vertex subsets, which suggests that $A$ is computable in polynomial time on graphs of bounded pathwidth.
\end{problem}

\begin{problem}
Study graph likelihood for a general attachment kernel, replacing the uniform
choice of $k$ by a distribution $\nu_t$ on $\{0,\dots,t-1\}$ as in
\cite{JS13}.  Assuming $\nu_1(0)=1$, the subset dynamic program extends by
putting $F_\nu(\{v\})=1$ and
\[
F_\nu(S)=\lvert S\rvert\sum_{v\in S}
\frac{\nu_{\lvert S\rvert}(\deg_{G[S]}v)F_\nu(S\setminus\{v\})}
{\bin{\lvert S\rvert-1}{\deg_{G[S]}v}},
\qquad
\Lk_\nu(G)=\frac{F_\nu(V(G))}{n!\,\lvert\Aut(G)\rvert}.
\]
Determine how the extrema and entropy depend on $\nu$.  The same recurrence
turns $\Lk_\nu$ into an exactly computable marginal likelihood of an
unlabelled graph, which is the quantity that likelihood-based inference for
mechanistic network models must otherwise approximate.
\end{problem}

\subsection*{Computational verification}

All computations use exact integer or rational arithmetic, with \texttt{geng}
and \texttt{nauty} \cite{nauty} for graph generation and automorphism groups.
\paragraph{Improved exact pruning and current measurements.}
The minimum-search implementation applies structural lower bounds
before evaluating the full recurrence.
Its initial lower bound $B_0(H)$ for $A(H)$ combines the universal bound
with a relaxation preserving a twin class and its common-neighbourhood
size.  Further bounds are obtained from the deletion recurrence:
\[
B_{r+1}(H)=\sum_{v\in V(H)}
\frac{B_r(H-v)}{\binom{|V(H)|-1}{\deg_Hv}}.
\]
The implementation evaluates $B_0$, $B_1$ and $B_2$, and takes
$B(H)=\max\{B_0(H),B_1(H),B_2(H)\}$.  Each term is a lower bound for
$A(H)$; vertex deletions can create additional twins and strengthen the
initial structural bound.  For an incumbent likelihood $\lambda$, a graph
is rejected only when
\[
B(H)>|V(H)|!\lvert\Aut(H)\rvert\lambda.
\]
All comparisons are exact and equality is retained, so pruning preserves
every possible minimizer.  Pre-generation bounds preserve the sizes and
internal signs of one or two twin classes, relax the remaining adjacencies,
and minimize over their possible neighbourhoods.  These signed twin-family
bounds raise the necessary automorphism thresholds and exclude families whose
maximum possible automorphism order cannot meet the required threshold.  The main saving is
that only surviving graphs require the full subset or blow-up recurrence.
Bounded evaluation batches control memory, and running-minimum logs retain
all ties for the final canonical audit.  The current native pruning helper
supports orders through $18$.  It screens minimum searches; computing a
maximum or an individual exact likelihood requires the corresponding
separate evaluation procedure.

Table~\ref{tab:smallpruning} reports the latest complete minimum-only reruns
on an Apple M4 Max with $14$ physical cores and $36$ GiB of RAM, using
Apple Clang 21.0.0, \texttt{-O2}, \texttt{nauty} and GMP.  Each row is one
complete measured run of the strengthened algorithm, with exact evaluation
of two survivors.  The known minimum and both minimizing isomorphism
classes, the balanced complete bipartite graph and its complement, are
recovered at every order.  Times are native-process CPU seconds, including
startup and both pruning and survivor evaluation; separately measured graph
generation and Python coordination are excluded.  They are not elapsed
wall times or measurements of the maximum-likelihood search.

\begin{table}[htbp]
\centering
\begin{tabular}{@{}rrrr@{}}
\toprule
$n$ & input graphs & exact evaluations & pruned CPU seconds\\
\midrule
$6$ & $156$ & $2$ & $0.0062$\\
$7$ & $1\,044$ & $2$ & $0.0069$\\
$8$ & $12\,346$ & $2$ & $0.0146$\\
$9$ & $274\,668$ & $2$ & $0.1740$\\
$10$ & $12\,005\,168$ & $2$ & $7.5689$\\
\bottomrule
\end{tabular}
\caption{Latest complete minimum searches with strengthened exact pruning
on the Apple M4 Max.  Each entry is a single complete run, including exact
evaluation and logging of all tied minima.}
\label{tab:smallpruning}
\end{table}

The order-$10$ rerun provides a direct comparison on the same complete input:
evaluating all $12\,005\,168$ graphs took $175.3315$ CPU seconds, whereas
pruning followed by exact evaluation of the two survivors took $7.5689$
CPU seconds, a factor of about $23.2$.  Graph generation took a further
$2.6826$ CPU seconds and is excluded from both figures.  Direct evaluation
used batches of at most $8192$ graphs to bound memory.  For orders $6$ and
$7$, process startup costs more than the saved evaluation work; this does
not affect the much larger reduction at order $10$.  Raw per-batch records,
source identities and a reproduction driver are retained in the computational records.
Matched-machine residual profiles are also retained, with their
sample coverage stated explicitly; they are not substituted for complete
search timings.

The complete minimum-only reruns at orders $11$ through $16$ use the same
Mac and strengthened pruning, together with the canonical signed twin-family
reduction.  At quotient orders $q\le9$, all eligible unlabelled bases are
coloured by \texttt{vcolg}; at $q\ge10$, a pinned complete catalogue of subgroup
conjugacy classes supplies the orbital cover.  Catalogue archive hashes,
subgroup census checks, group orders and the complete range plans are
retained in the computational records.  Within the fixed vertices of an action, the
implementation chooses between all pair-orbital choices and all unlabelled
fixed-point graphs with every compatible colour palette.  Both choices
cover the same possibilities up to relabelling; the smaller representation
space is used.  Necessary family-specific automorphism thresholds and the
canonical twin checks are applied before graphwise pruning.  These generation
changes, as well as fewer exact evaluations, contribute to the reduction in
runtime; a residual-sample pruning ratio alone would not measure it.

Table~\ref{tab:fullreruns} gives one complete measured search per order,
using one low-priority pipeline worker sequentially.  The incumbent is an
exactly evaluated witness, so strict pruning retains every potential
improvement and equality.  Every survivor is evaluated, all tied minima
are logged, and their canonical isomorphism classes are closed under
complementation. Each order has exactly two minimizing classes; at order
$15$ these are
$C_5[\bar K_3]$ and its complement.  The families column counts viable
signed families modulo complement, including the twin-free layer; evaluation
counts are representations and may include subgroup duplicates.

\begin{table}[htbp]
\centering
\begin{tabular}{@{}rrrrr@{}}
\toprule
$n$ & viable families & exact evaluations & native CPU s & elapsed s\\
\midrule
$11$ & $31$ & $2$ & $0.498$ & $0.639$\\
$12$ & $21$ & $1$ & $0.141$ & $0.224$\\
$13$ & $126$ & $20$ & $4.533$ & $4.588$\\
$14$ & $67$ & $9$ & $0.574$ & $0.715$\\
$15$ & $237$ & $113$ & $18.760$ & $17.701$\\
$16$ & $276$ & $113$ & $14.708$ & $14.624$\\
\bottomrule
\end{tabular}
\caption{Complete minimum-only reruns with the improved generation and exact
pruning on the Apple M4 Max, using one pipeline worker.  Each row includes
all survivor evaluations and complete minimum/tie logging.}
\label{tab:fullreruns}
\end{table}

Elapsed times include colouring or orbital expansion, canonical twin and
automorphism checks, pruning, all exact survivor evaluations, log writes and
durable checkpoints.  The CPU column sums the native child processes and
excludes Python coordination.  The six searches total $38.49$ elapsed
seconds.  Shared preflight took a further $15.56$ seconds to recompute all
exact family bounds and check the cached catalogue and binary identities;
prior construction of catalogue metadata, group orders and the small-base
files is excluded.  Final certification took $2.59$ additional seconds:
every completed range, survivor count and minimum log was audited, and an
independent rational subset recurrence recomputed each final likelihood,
using analytic automorphism orders after verifying the witness structures.
No failed or incomplete shard remains.  The smaller order-$12$ workload
reflects stronger family exclusions; none of these minimum-only timings
measures the maximum-likelihood search.

\paragraph{Measured complete order-$17$ search.}
The order-$17$ campaign used the same improved generation, exact pruning
and complete minimum/tie logging, with $744$ viable families including
the twin-free layer. Table~\ref{tab:run17} reports the measured completed
work on the Mac and four Linux hosts. It evaluated $170\,032\,082$
representations; duplicate subgroup representations are included.

\begin{table}[htbp]
\centering
\small
\begin{tabular}{@{}lrrrrr@{}}
\toprule
host & cap & tasks & evaluations & \shortstack{native\\CPU h} & \shortstack{host\\elapsed h}\\
\midrule
Mac & $8\to4$ & $230\,636$ & $81\,397\,676$ & $513.733$ & $104.22$\\
Linux 1 & $12$ & $136\,717$ & $48\,844\,199$ & $1053.791$ & $55.63$\\
Linux 2 & $6$ & $33\,501$ & $11\,474\,180$ & $340.541$ & $61.99$\\
Linux 3 & $8$ & $62\,502$ & $22\,646\,493$ & $557.850$ & $80.82$\\
Linux 4 & $2$ & $16\,626$ & $5\,669\,534$ & $151.899$ & $80.59$\\
total & $36\to32$ & $479\,982$ & $170\,032\,082$ & $2617.814$ & $383.25$\\
\bottomrule
\end{tabular}
\caption{Complete order-$17$ minimum search. Host elapsed time runs from
that host's initial controller launch to its final worker exit. The total
row sums the displayed per-host entries, including elapsed times. The
Mac cap was reduced from eight workers to four during the campaign.}
\label{tab:run17}
\end{table}

The search ran from October 23 at 15:20 to October 27 at 23:33 EDT, 2026.
Its $104.22$ calendar hours include memory holds, checkpoint-safe
controller reloads, changing worker availability and contention from
unrelated jobs; they are not a benchmark at a constant worker count.
The $2617.814$ native CPU hours sum committed generation, canonical twin
and automorphism filtering, exact pruning and survivor evaluation.
They exclude Python coordination and uncommitted or replayed attempts.
Cached catalogue and base construction, earlier pilots and final
certification are excluded from these search timings.

Native checkpoints cover at most $65\,536$ higher-$q$ representations,
and exact evaluation uses batches of at most $512$ survivors. All minimum
logs and restart checkpoints are retained. Preflight tests checked
uncommitted replay, preservation of committed cursors, refusal to recover
live owners, idempotent finished replay and duplicate minimum logging on
all five hosts. The final full checkpoint audit and independent witness
checks establish Proposition~\ref{prop:min17}.

\paragraph{Measured complete order-$18$ search.}
The same improved generation and pruning completed all $829$ viable
families at order $18$ on October 23, 2026. Table~\ref{tab:run18} reports
actual completed runs on the Mac and four Linux hosts; their local queue
filesystems are separate. The worker caps sum to $36$ pipeline workers,
with native stages running at low priority. These are measured full-search
timings, not extrapolations from pilots.

\begin{table}[htbp]
\centering
\begin{tabular}{@{}lrrrrr@{}}
\toprule
host & cap & shards & evaluations & \shortstack{native\\CPU h} & \shortstack{elapsed\\min}\\
\midrule
Mac & $8$ & $18\,069$ & $486\,585$ & $11.500$ & $65.96$\\
Linux 1 & $12$ & $15\,319$ & $425\,496$ & $30.794$ & $92.23$\\
Linux 2 & $6$ & $4\,181$ & $107\,067$ & $10.776$ & $112.68$\\
Linux 3 & $8$ & $6\,512$ & $167\,906$ & $15.381$ & $112.83$\\
Linux 4 & $2$ & $1\,067$ & $28\,301$ & $2.510$ & $47.97$\\
total & $36$ & $45\,148$ & $1\,215\,355$ & $70.961$ & $431.67$\\
\bottomrule
\end{tabular}
\caption{Complete order-$18$ minimum search with all-minimum logging.
Host elapsed time runs from controller launch to its last worker exit.
The total row sums the displayed per-host entries, including elapsed times.
Evaluations count representations, not distinct graphs.}
\label{tab:run18}
\end{table}

The complete order-$18$ campaign took $112.84$ minutes from first launch
to last exit. The native CPU totals include generation, canonical twin and
automorphism filtering, exact pruning and survivor evaluation; they exclude Python
coordination. Elapsed times include orchestration and durable checkpoint
and minimum-log writes. The common cached subgroup catalogue, group orders,
small bases, prior pilots and final certification are outside this search
timing. Native children have $128$ MiB Linux address-space limits and $600$
CPU-second limits, and evaluations use batches of at most $512$ graphs.
On the Mac, where a finite address-space limit is unsupported, bounded
inputs and memory admission protect the documented reserve. Existing
unrelated computations were left running. No failed or incomplete shard
remains. Restart/replay and split-range audits pass; final coverage,
survivor-hash, exact-count and all-minimum-log audits establish
Proposition~\ref{prop:min18}.

\paragraph{Exhaustive extrema and independent evaluations.}
The minimum-only searches above do not compute maximum likelihoods.
The exhaustive order-$12$ certificate covers both extrema; its full-run
timing is retained separately because no replacement full-extrema
measurement is available.

The primary exhaustive certificate for Theorem~\ref{thm:small} is an
independently developed C plugin compiled directly into \texttt{geng}.  It
uses the exact automorphism-group order returned by \texttt{nauty}, checked
unsigned $128$-bit arithmetic for the integer-scaled subset recurrence, and
exact $192$-bit cross-products for rational likelihood comparisons.  It
rejects orders above $12$ rather than risk overflow. The program reproduces
the extrema through order $11$ as well as
certifying the order-$12$ result.

At order $12$, complement symmetry reduces the computation to one
representative from each of the $82\,545\,586\,656$ complement orbits of the
$165\,091\,172\,592$ unlabelled graphs; the latter count includes $720$
self-complementary graphs.  The computation was divided into $4\,800$ exact
\texttt{geng} residue shards and run with $14$ workers.  All shard outputs
completed, with no failed shards or nonempty error files.  Their equivalent
full graph counts sum to $165\,091\,172\,592$ before the merge reports the
global extrema.  Aggregate helper CPU time was about $817$ core-hours, and
elapsed wall time was approximately $68$ hours.  These are measurements of
the earlier integrated exhaustive-extrema algorithm, without the newer
minimum-only pruning; no replacement full-extrema timing is available.

The principal implementation for individual larger graphs performs graph
processing in the Wolfram Language and invokes a C helper that computes the
scaled recurrence modulo multiple $64$-bit primes; the exact result is
reconstructed by the Chinese remainder theorem.  It returns
\[
A(C_5[\bar K_3])
=\frac{63977511069907}{427822618422142944000000},
\]
agreeing with \eqref{eq:blowup} evaluated over the $4^5$ vectors
$m'\le(3,3,3,3,3)$ and with the direct rational recurrence over the $2^{15}$
subsets.

For the extended $C_5[m]$ calculation, a separate arbitrary-precision
implementation enumerates positive five-part vectors modulo the dihedral
symmetries of $C_5$ and evaluates \eqref{eq:blowup} with exact rationals.  It
reproduces the ratios in Table~\ref{tab:c5more}, exhaustively checks every
positive vector through order $50$, and evaluates every nearly balanced
vector through order $100$.  Independent exact evaluations in the Wolfram
Language reproduce the resulting ratios at orders $26$, $30$ and $50$.

The degree-sequence calculation enumerates all graphical sequences and applies
Lemma~\ref{lem:screen} by integer cross multiplication.
For its second deletion, the neighbours of each possible deleted vertex are
assigned among the remaining degree classes so as to minimise the resulting
lower bound; allowing the minimizing assignment independently for each
vertex is a relaxation and therefore cannot discard a genuine competitor.

\end{document}